\documentclass[a4paper,12pt]{amsart} 

\usepackage{amsmath,amsfonts,amssymb,amsthm,mathtools} 
\mathtoolsset{showonlyrefs=true} 
\usepackage{enumitem}
\usepackage{ytableau}
\usepackage{tikz-cd}
\usetikzlibrary{decorations.pathmorphing} 

\usepackage{hyperref} 

\DeclareMathOperator{\initial}{in}
\DeclareMathOperator{\Gr}{Gr}
\DeclareMathOperator{\Proj}{Proj}
\newcommand{\la}{\lambda}
\newcommand{\bC}{\mathbb C}
\newcommand{\bP}{\mathbb P}
\newcommand{\bZ}{\mathbb Z}
\newcommand{\bN}{\mathbb N}

\newcommand{\cL}{\mathcal L}
\newcommand{\cJ}{\mathcal J}
\newcommand{\cP}{\mathcal P}
\newcommand{\cI}{\mathcal I}
\DeclareMathOperator{\grad}{grad}
\renewcommand{\ss}{\mathrm{ss}}
\renewcommand{\mod}{\mathrm{mod}}
\newcommand{\pbw}{\mathrm{pbw}}

\newtheorem{innercustomthm}{Theorem}
\newenvironment{customthm}[1]
  {\renewcommand\theinnercustomthm{#1}\innercustomthm}
  {\endinnercustomthm}

\theoremstyle{definition}
\newtheorem{theorem}{Theorem}[section]

\newtheorem{corollary}[theorem]{Corollary}
\newtheorem{lemma}[theorem]{Lemma}

\newtheorem{definition}[theorem]{Definition}
\newtheorem{proposition}[theorem]{Proposition}

\newtheorem{example}[theorem]{Example}
\newtheorem{remark}[theorem]{Remark}

\mathcode`<="013C
\let\relprec\prec
\renewcommand{\prec}{{\relprec}}
\let\relll\ll
\renewcommand{\ll}{{\relll}}

\title[Beyond the Sottile-Sturmfels degeneration]{Beyond the Sottile-Sturmfels degeneration of a semi-infinite Grassmannian}

\author{Evgeny Feigin}
\address{E. Feigin:\newline
HSE University\\
Faculty of Mathematics\\
Ulitsa Usacheva 6\\Moscow 119048\\Russia\newline
{\it and }\newline
Skolkovo Institute of Science and Technology\\ 
Center for Advanced Studies\\
Bolshoy Boulevard 30, bld. 1\\
Moscow 121205\\
Russia
}
\email{evgfeig@gmail.com}
\author{Igor Makhlin}
\address{I. Makhlin:\newline
Skolkovo Institute of Science and Technology\\ 
Center for Advanced Studies\\
Bolshoy Boulevard 30, bld. 1\\
Moscow 121205\\
Russia\newline
{\it and }\newline
HSE University\\
Faculty of Mathematics\\
Ulitsa Usacheva 6\\Moscow 119048\\Russia}
\email{imakhlin@mail.ru}
\author{Alexander Popkovich}
\address{A.Popkovich:\newline
HSE University\\
Faculty of Mathematics\\
Ulitsa Usacheva 6\\Moscow 119048\\Russia}
\email{aspopkovich@yandex.ru}

\thanks{All authors were partially supported by the grant RSF 19-11-00056. 
I.\ Makhlin was supported in part by the ``Young Russian Mathematics'' contest.}

\begin{document}

\begin{abstract}
We study toric degenerations of semi-infinite Grassmannians (a.k.a.\ quantum Grassmannians). While the toric degenerations of the classical Grassmannians are well studied, the only known example in the semi-infinite case is due to Sottile--Sturmfels. We start by providing a new interpretation of the Sottile--Sturmfels construction by finding a poset such that their degeneration is the toric variety of the order polytope of the poset. We then use our poset to construct and study a new toric degeneration in the semi-infinite case. Our construction is based on the notion of poset polytopes introduced by Fang--Fourier--Litza--Pegel. As an application we introduce semi-infinite PBW-semistandard tableaux, giving a basis in the homogeneous coordinate ring of a semi-infinite Grassmannian.
\end{abstract}

\maketitle

\section*{Introduction}

Let $G$ be a simple Lie group and let $P$ be a parabolic subgroup. The flag variety
$G/P$ enjoys the Pl\"ucker embedding into the product of projectivizations of fundamental 
representations. The corresponding multi-homogeneous coordinate ring is
known as the Pl\"ucker algebra. In particular, in type $A$ this algebra is defined by
the classical quadratic Pl\"ucker relations ~\cite{Fulton}.

In order to pass to the semi-infinite setting one considers the Drinfeld-Pl\"ucker data.
Roughly speaking, this amounts to replacing the field of complex numbers with the ring
of power series in one variable (see \cite{MF}), a similar construction is used in the theory of arc schemes \cite{Mu}. The semi-infinite flag
variety then sits in the product of projectivizations of fundamental 
representations over the power series. 
The corresponding multi-homogeneous coordinate ring
turns out to be very meaningful from the point of view of representation theory.

In this paper we are interested in the case when $G=SL_n(\bC)$ and $P$ is maximal parabolic. Such varieties are known as \textit{semi-infinite Grassmannians} \cite{FeFr,MF} or, alternatively, as \textit{quantum Grassmannians}~\cite{Sot,SS}. The latter terminology comes from literature studying rational curves in classical Grassmannians, such curves are shown to correspond to points in quantum Grassmannians. 

While a rich and diverse theory exists on the subject of flat degenerations and, in particular, toric degenerations of classical flag varieties and Grassmannians (see~\cite{FaFL} for a detailed survey), results in the semi-infinite case remain scarce. Constructions of certain non-toric Gr\"obner degenerations of semi-infinite flag varieties are implicit in the works~\cite{FM2,M} and the only known toric degeneration of the semi-infinite Grassmannian is due to~\cite{SS}. The goal of this paper is to explicitly construct a second degeneration of this kind as well as to present a uniform combinatorial method that works equally well in the finite and semi-infinite cases. Our interest is partially motivated by the connection between semi-infinite Pl\"ucker algebras and global Weyl and Demazure modules \cite{BF,FM2,DF,DFF}. 

Our approach is to combine the results in~\cite{SS} with certain order theoretic tools which have proved useful in the finite case. Namely, the two most explicit and well-studied degenerations of Grassmannians are known to be toric varieties of \textit{poset polytopes}: polytopes associated with a poset, a notion introduced in~\cite{Stanley} and generalized in~\cite{ABS,FaF,FFLP}. The first of these two degenerations is originally due to~\cite{S} with its key properties obtained in~\cite{GL,KM}. In particular, in~\cite{GL} this degeneration was realized as the \textit{Hibi variety}~\cite{H} of a distributive lattice $\mathcal L$. However, the Hibi toric variety is known to correspond to the order polytope (a type of poset polytope) of the poset $Q$ whose lattice of order ideals is isomorphic to $\mathcal L$ (see~\cite{wiki}). This poset $Q$ is described explicitly in~\cite{FL}.

The second degeneration is the toric variety of the FFLV polytope~\cite{FFL1} originally obtained in~\cite{FFL2} and made explicit in~\cite{FFFM}. An observation due to~\cite{ABS,FL} is that this toric variety corresponds to the chain polytope (another type of poset polytope) of the same poset $Q$. Now, a toric degeneration of the semi-infinite Grassmannian is obtained in~\cite{SS} as the Hibi variety of another distributive lattice $\tilde{\mathcal L}$. Motivated by the finite case we propose to find the poset $\tilde Q$ corresponding to $\tilde{\mathcal L}$ in the mentioned sense and to look for degenerations of the semi-infinite Grassmannian among the toric varieties of other poset polytopes of $\tilde Q$. One such degeneration is indeed found, the corresponding polytope is, however, not an order or a chain polytope but a more general kind of poset polytope defined in~\cite{FaF,FFLP}. Here it is worth mentioning that the degeneration in~\cite{SS} restricts to the degeneration in~\cite{S,GL,KM} (one has an embedding of flat families) and the same holds for the degeneration constructed here and the degeneration in~\cite{FFL2,FFFM}.

Recall that  the Grassmannian is the $\Proj$ of the Pl\"ucker algebra $\cP$ and the semi-infinite Grassmannian is the $\Proj$ of the semi-infinite Pl\"ucker algebra $\tilde\cP$. Hence, flat degenerations of these varieties can be obtained as $\Proj$ of initial algebras of $\cP$ and $\tilde\cP$. In this paper we successively discuss the four mentioned degenerations (two finite and two semi-infinite), all in similar fashion. In each case we first define a monomial order such that the initial terms of the standard generators of $\cP$ or $\tilde\cP$ generate the respective initial algebra $\mathcal A$ (i.e.\ the generators form a \textit{sagbi basis}). Next we consider a poset and one of its poset polytopes $\Pi$ such that the toric ring of $\Pi$ is isomorphic to $\mathcal A$. This implies that the toric variety of $\Pi$ is a flat degeneration of the (semi-infinite) Grassmannian. After that we use general properties of poset polytopes to obtain a basis in the (semi-infinite) Pl\"ucker algebra. The basis is parametrized by increasing sequences of order ideals in the poset which in each of the cases are visualized as tableaux satisfying some sort of semistandardness condition.

Results concerning the finite case (Section~\ref{finite}) are known and are provided for mostly expository purposes. Section~\ref{SS} deals with the first semi-infinite degeneration, the results here are mostly interpretations of those in~\cite{SS}, however, we contribute by finding a poset $\tilde Q$ such that the following holds.
\begin{customthm}{A}[cf.\ Lemma~\ref{infssashibiring}, Corollary~\ref{ssdegen}, Proposition~\ref{posetisom}]
The toric degeneration constructed in~\cite{SS} is the toric variety of the order polytope of $\tilde Q$. The lattice $\tilde{\mathcal L}$ discussed in~\cite{SS} is isomorphic to the lattice of order ideals in $\tilde Q$.
\end{customthm}
Section~\ref{new} is devoted to the construction of our new degeneration, the key results can be summed up as follows.

\begin{customthm}{B}[cf.\ Theorem~\ref{main1}, Lemma~\ref{infpbwashibiring}, Corollary~\ref{newdegen}]
There exists a monomial order and a poset polytope $\Pi$ of $\tilde Q$ (different from the order polytope) such that the initial algebra is isomorphic to the toric ring of $\Pi$. Consequently, the toric variety of $\Pi$ is also a flat degeneration of the semi-infinite Grassmannian.
\end{customthm}

\begin{customthm}{C}[cf.\ Proposition~\ref{infpbwposetisom}, Theorem~\ref{main2}]\label{C}
There is a basis in $\tilde\cP$ enumerated by weakly increasing sequences of order ideals in $\tilde Q$ or, equivalently, by \emph{PBW-semistandard semi-infinite tableaux} (a generalization of PBW-semistandard Young tableaux~\cite{Fe}).
\end{customthm}

We point out that Theorem~\ref{C} has one immediate representation theoretic application. The homogeneous components of $\tilde\cP$ are known to be the restricted duals of global Weyl modules with highest weight $m\omega$ with $\omega$ fundamental~\cite{BF,Kat,FM2}. Hence we obtain a basis and a character formula for these modules.

\section{Initial subalgebras and sagbi degenerations}

In this section we will recall some basic notions from Gr\"obner theory. Let $R = \mathbb{C}[x_1,\dots]$ be a ring of polynomials in finitely or countably many variables. 
\begin{definition}
A monomial order $<$ on $R$ is a total order on monomials in $R$ such that for any $x_i$ and monomials $m_1$,$m_2$ with $m_1 < m_2$ we have $m_1 \cdot x_i < m_2 \cdot x_i$.
\end{definition}

For $f \in R$ and a monomial order $<$ denote by $\initial_<(f)$ the initial term of $f$ with respect to $<$, i.e.\ the $<$-maximal term occurring in $f$. For a subalgebra $A \subset R$ the \textit{initial subalgebra} $\initial_<(A) \subset R$ is the subalgebra of $R$ spanned by the initial terms of all elements of $A$. 
\begin{definition}
For $A$ and $<$ as above a generating set $\{s_k\}_{k\in K}$ of $A$ is called a \textit{sagbi basis} of $A$ if the set $\{\initial_<(s_k)\}_{k\in K}$ generates $\initial_<(A)$.
\end{definition}

Sagbi stands for ``Subalgebra Analogue of Gr\"obner Bases for Ideals", this notion is due to~\cite{RS}. Note that since we will be working with countably generated algebras, we drop the standard requirement for sagbi bases to be finite. 

\begin{remark}
It is easily seen that the initial subalgebra $\initial_<(A)$ can be viewed as an associated graded algebra of $A$. There exists a newer notion of \textit{Khovanskii bases} (\cite{KaM}) which generalizes sagbi bases to general filtered algebras. However, in this paper we only consider subalgebras of polynomial algebras for which the corresponding filtration components are spanned by monomials which makes the language of sagbi bases more natural for our goals.
\end{remark}

There are certain natural ways to construct a monomial order using the chosen enumeration of variables $x_1,\dots$. The best known is the lexicographic order, however, we will repeatedly use the so-called degree reverse lexicographic order.
\begin{definition}
A \emph{degree reverse lexicographic order} is a monomial order which works as follows: let $m_1 = \prod_i x_i^{a_i}$ and $m_2 = \prod_i x_i^{b_i}$. Then $m_1 < m_2$ if and only if $\sum\limits_i a_i < \sum\limits_i b_i$ or  $\sum\limits_i a_i = \sum\limits_i b_i$ and the last non-zero entry in $a-b$ is positive.
\end{definition}

We work with degree reverse lexicographic orders in order to follow the conventions in~\cite{SS} and apply their results directly. However, these orders have a certain disadvantage (compared to lexicographic orders, for instance): when the number of variables is infinite they are not Artinian, this prohibits us from arguing directly by induction on the order. As a workaround we will assume in this section that the subalgebra $A\subset R$ is homogeneous with respect to some $\mathbb N_0^l$-grading $\grad$ on $R$ and has finite-dimensional components (we only consider gradings on polynomial rings for which all variables are homogeneous elements). The subalgebras studied in this paper indeed have this property.

One useful property of sagbi bases is that they allow bases to be lifted from an initial subalgebra to the subalgebra itself.
\begin{proposition}\label{liftbasis}
For a monomial order $<$ consider a sagbi basis $\{s_k\}_{k\in K}$ of $A$ consisting of $\grad$-homogeneous elements. For every $j\in\mathbb N_0$ consider a set of natural numbers $d^j=\{d^j_k\}_{k\in K}$ with finitely many nonzero elements. Suppose that $\{\prod_{k\in K} \initial_<(s_k)^{d^j_k}\}_{j\in\mathbb N_0}$ is a basis in $\initial_< A$. Then $\{\prod_{k\in K} s_k^{d^j_k}\}_{j\in\mathbb N_0}$ is a basis in $A$.
\end{proposition}
\begin{proof}
Since \[\prod_{k\in K} \initial_<(s_k)^{d^j_k}=\initial_<\prod_{k\in K} s_k^{d^j_k},\] the linear independence follows from the straightforward general fact that a set of polynomials is linearly independent if their initial parts are. For the spanning property consider a $\grad$-homogeneous $f\in A$. We must have $c\prod_{k\in K} \initial_<(s_k)^{d^j_k}=\initial_< f$ for some $j$ and $c\in\bC$. Then for $f'=f-c\prod_{k\in K} s_k^{d^j_k}$ we have $\initial_<f'<\initial_< f$. This procedure can be iterated and will terminate because the $\grad$-homogeneous component containing $f$ is finite-dimensional.
\end{proof}

An initial subalgebra of $A$ must itself be $\grad$-homogeneous. Moreover, from Proposition~\ref{liftbasis} one deduces the key fact that it ``has the same size''.
\begin{corollary}\label{samesize}
For any monomial order $<$ the subalgebras $A$ and $\initial_< A$ have the same (multivariate) Hilbert series with respect to $\grad$.
\end{corollary}
\begin{proof}
One may always find a sagbi basis as in Proposition~\ref{liftbasis} by choosing a large enough set of $\grad$-homogeneous generators.
\end{proof}

Let us also briefly recall an alternative way of obtaining initial parts and subalgebras. Choose a sequence $\la=(\la_1,\dots)$ of real numbers with one element for every $x_i$. This defines a \textit{$\la$-weight} for every monomial $x_{i_1}\dots x_{i_m}\in R$ equal to $\la_{i_1}+\dots+\la_{i_m}$. The initial part $\initial_\la f$ of $f\in R$ with respect to $\la$ is the sum of those terms in $f$ which have the largest occurring $\la$-weight. The initial subalgebra $\initial_\la A$ is again spanned by the initial parts of elements of $A$. Sagbi bases with respect to $\la$ are defined similarly. The reason for us to consider this notion is the following well-known fact.
\begin{theorem}\label{flatfamily}
Let $\la$ be as above. Then there exists a flat $\bC[t]$-algebra $\tilde A$ such that $\tilde A/\langle t\rangle\simeq \initial_\la A$ while for any nonzero $c\in\bC$ we have $\tilde A/\langle t-c\rangle\simeq A$. 
\end{theorem}
In the case when $R$ is finitely generated and $A$ has a finite sagbi basis with respect to $\la$, an explicit construction of the algebra $\tilde A$ can be found in, for instance,~\cite[Corollary III.17]{kim}. The same argument applies verbatim without the finiteness assumptions. A way to rephrase this theorem is to say that 
we have a flat sheaf of algebras on $\mathbb A^1$ with its fiber at $0$ isomorphic to $\initial_\la A$ and all other fibers isomorphic to $A$, i.e.\ 
$\initial_\la A$ is a flat degeneration of $A$.

In the below text, however, we will be working with initial subalgebras and sagbi bases with respect to (degree reverse lexicographic) monomial orders rather than real weights. Hence we now explain how the flat family construction can be adapted to this situation.
\begin{lemma}\label{revlexflat}
Let $<$ be the degree reverse lexicographic monomial order with respect to the chosen enumeration. Suppose that $A$ is homogeneous with respect to total degree and has a sagbi basis $\{s_k\}_{k\in K}$ with respect to $<$ such that every $s_k$ has total degree no greater than some fixed integer $M$. Then there exists $\la=(\la_1,\dots)$ such that $\initial_\la A=\initial_< A$.
\end{lemma}
\begin{proof}
We may assume that all $s_k$ are homogeneous with respect to total degree and $\grad$. Now set $\la_i=-(M+1)^i$. This ensures that for every $j$ we have $\initial_\la s_k=\initial_< s_k$, hence $\initial_< A\subset\initial_\la A$. Now suppose the reverse inclusion does not hold and for some $f\in A$ we have $\initial_\la(f)\notin\initial_< A$. Suppose that $f$ contains terms which are elements of $\initial_< A$ and of these let $m$ have the greatest $\la$-weight. We can choose a $\grad$-homogeneous $g\in A$ with $\initial_\la g=m$ and consider $f'=f-g$. By iterating this procedure we arrive at a nonzero polynomial in $A$ which contains no term lying in $\initial_< A$ which is not possible.
\end{proof}

\begin{remark}
With the use of more involved arguments such a $\la$ may also be obtained for much more general pairs of $A$ and $<$. For instance, the existence of a finite sagbi basis is always sufficient.
\end{remark}

For us the main takeaway is that by combining Lemma~\ref{revlexflat} with Theorem~\ref{flatfamily} we obtain a flat degeneration in the case of such a monomial order. We use a geometric wording this time because the rings we will be working with are geometric in nature.
\begin{corollary}\label{geomflatfam}
Let $A$ and $<$ satisfy the conditions of Lemma~\ref{revlexflat} and suppose that $A$ is homogeneous with respect to a chosen $\mathbb N_0$-grading. Then there exists a flat projective scheme over $\mathbb A^1_\bC$ with fiber over $0$ isomorphic to $\Proj(\initial_< A)$ and all other fibers isomorphic to $\Proj(A)$.
\end{corollary}
We will refer to flat degenerations of this form as \textit{sagbi degenerations}.

\section{Posets and polytopes}\label{polytopes}

In this section we discuss interpolating polytopes, a family of polytopes associated with a poset. We use the conventions and terminology from~\cite{Makhlin} which should be viewed as a variation of those used earlier in~\cite{FaF,FFLP}.

Let $(P,\prec)$ be a (possibly infinite) poset. Let $\mathcal J(P,\prec)$ denote the set of all \textbf{finite} order ideals in $(P,\prec)$, i.e.\ subsets $J\subset P$ for which $p\in J$ and $q\prec p$ imply $q\in J$. Consider a partition $O\sqcup C=P$. Let $\Omega$ denote the subspace of $\mathbb R^P$ of points with finitely many nonzero coordinates.
\begin{definition}\label{polytopedef}
For an order ideal $J\in\mathcal J(P,\prec)$ let $v_{O,C}(J)\in\Omega$ denote the point with
\begin{itemize}
\item $v_{O,C}(J)_p=1$ if $p\in J\cap O$ or $p\in\max_\prec(J)\cap C$ and
\item $v_{O,C}(J)_p=0$ otherwise,
\end{itemize}
where $\max_\prec(J)$ is the set of $\prec$-maximal elements in $J$. Let $\Pi_{O,C}(P,\prec)\subset\Omega$ be the convex hull of all points $v_{O,C}(J)$ with $J\in\mathcal J(P,\prec)$, this is the \emph{interpolating polytope} given by the partition $O\sqcup C$.
\end{definition}

\begin{remark}
Polytopes of this kind are usually defined by a set of linear inequalities, in particular, such is the definition found in~\cite[Section 4]{Makhlin}. For the equivalence to the above definition in the case of finite $P$ see Proposition 4.3 therein. We define the polytope as a convex hull since this is simpler and suffices for our needs.
\end{remark}

In the extreme cases when $O=P$ and $C=P$ we obtain, respectively, the order polytope and the chain polytope of $(P,\prec)$ that were introduced in~\cite{Stanley}. The other cases can be said to interpolate between them. A key property of these polytopes is as follows.
\begin{theorem}\label{intpoints}
For every $k\in\mathbb N_0$ and every integer point $u\in k\Pi_{O,C}(P,\prec)$ (the $k$-fold dilation of $\Pi_{O,C}(P,\prec)$) there exists a unique weakly increasing tuple $J_1\subset\dots\subset J_k$ in $\mathcal J(P,\prec)$ such that \[u=v_{O,C}(J_1)+\dots+v_{O,C}(J_k).\] 
\end{theorem}
\begin{proof}
For the case of a finite $P$ see~\cite[Proposition 4.6]{Makhlin} and its proof. If $P$ is countable, let $J\in\cJ(P,\prec)$ be the minimal order ideal containing all $p$ with $u_p\neq 0$, the statement reduces to the case of the finite poset $(J,\prec)$.
\end{proof}
One simple consequence of this theorem is that $\Pi_{O,C}(P,\prec)$ has no integer points other than the $v_{O,C}(J)$ and that these are pairwise distinct. It is also evident that all $v_{O,C}(J)$ are vertices of $\Pi_{O,C}(P,\prec)$, since they are all vertices of the unit cube. We are, however, primarily interested in the implications this has in toric geometry.


Denote $\bC[P]=\bC[\{y_p\}_{p\in P}]$ and for $u\in\Omega$ with non-negative integer coordinates denote \[\mathbf y^u=\prod_{p\in P}y_p^{u_p}\in\bC[P].\]

\begin{definition}\label{hibidef}
Consider the ring $\bC[P][s]$ with an $\mathbb N_0$-grading given by degree in $s$. The \textit{generalized Hibi ring} of $(P,\prec)$ is the ring $A_{O,C}(P,\prec)\subset \bC[P][s]$ generated by the monomials $s\mathbf y^{v_{O,C}(J)}$ for all $J\in\mathcal J(P,\prec)$. The \textit{generalized Hibi variety} of $(P,\prec)$ is $\Proj A_{O,C}(P,\prec)$ (with respect to the induced grading).
\end{definition}
Note that the generalized Hibi variety is by definition the toric variety of the polytope $\Pi_{O,C}(P,\prec)$. The terminology stems from the fact that when $P$ is finite and $O=P$ these objects are known as the \textit{Hibi ring} and the \textit{Hibi variety}, namesake of~\cite{H}.

Theorem~\ref{intpoints} provides a basis in the ring $A_{O,C}(P,\prec)$.
\begin{corollary}\label{hibibasis}
The monomials $s^k\mathbf y^{v_{O,C}(J_1)+\dots+v_{O,C}(J_k)}$ with  $J_1\subset\dots\subset J_k$ ranging over all finite weakly increasing tuples in $\mathcal J(P,\prec)$ form a basis in $A_{O,C}(P,\prec)$.
\end{corollary}

\section{The finite picture}\label{finite}

In this section we provide a brief overview of some results concerning degenerations of classical Grassmannians. These should be viewed as a prototype of the picture we wish to obtain in the semi-infinite case and also as an example illustrating the more involved semi-infinite constructions.


Fix integers $1\le p<n$ and consider a $p\times n$ matrix $Z$ whose elements are formal variables $z_{i,j}$ with $i\in[1,p]$ and $j\in[1,n]$. For an increasing tuple $I=(\alpha_1,\dots,\alpha_p)$ of integers in $[1,n]$ let $Z_I$ denote the $p\times p$ submatrix in $Z$ formed by columns $\alpha_1,\dots,\alpha_p$. Let $\cP\subset\bC[\{z_{i,j}\}]$ denote the subalgebra generated by all minors $D_I=\det(Z_I)$, a $\mathbb N_0$-grading on $\cP$ is given by setting the grading of every generator $D_I$ equal to $1$. $\cP$ is known as the \textit{Pl\"ucker algebra}, the variety $\Proj\cP$ is the Grassmannian $\Gr(p,n)$ of $p$-dimensional subspaces in $\bC^n$ (in other words, since $\cP$ is generated by $n\choose p$ elements, it defines a subvariety in $\bC\bP^{{n\choose p}-1}$, this is the image of $\Gr(p,n)$ under the \textit{Pl\"ucker embedding}). According to Corollary~\ref{geomflatfam}, initial subalgebras of $\cP$ provide flat degenerations of the Grassmannian, we discuss two such degenerations below.

\subsection{The Sturmfels/Gonciulea--Lakshmibai/Kogan--Miller degeneration}\label{ss}

Most of the results here are originally due to~\cite{S,GL,KM}. 

Consider the degree reverse lexicographic monomial order on $\bC[\{z_{i,j}\}]$ given by the ``row-wise" ordering of variables $z_{i,j}$:
\begin{equation}
z_{1,1}, z_{1,2}, \dots, z_{1,n}, z_{2,1}, \dots, z_{p,n}.
\end{equation}
We will denote this monomial order $<_{\ss}$ where ``ss'' stands for ``semistandard'' (see below). Note that for $I=(\alpha_1<\dots<\alpha_p)$ we have \[\initial_{<_{\ss}}D_I=(-1)^{n\choose 2}z_{1,\alpha_p}\dots z_{p,\alpha_1},\] i.e.\ the ``antidiagonal term''.
\begin{theorem}[see {\cite[Theorem 14.11]{CCA}}]\label{diagdegen}
The $n\choose p$ polynomials $D_I$ form a sagbi basis of $\cP$ with respect to $<_{\ss}$. 
\end{theorem}

Now let $(Q,\prec)$ be the poset consisting of elements $q_{i,j}$ with $1\le i\le p$ and $p+1\le j\le n$ where $q_{i_1,j_1}\preceq\ q_{i_2,j_2}$ if and only if $i_1\le i_2$ and $j_1\le j_2$ (and $\prec$ is the corresponding strict relation). Since we will not be considering other orders on $Q$, we denote $\cJ(Q)=\cJ(Q,\prec)$. 

Let $\mathcal I_{\ss}$ denote the set of increasing $p$-tuples of integers in $[1,n]$. It is easily seen that there is a bijection $\varphi_{\ss}:\mathcal I_{\ss}\to \cJ(Q)$ where for $I=(\alpha_1,\dots,\alpha_p)$ and all $i\in[1,p]$ the order ideal $\varphi_{\ss}(I)$ contains $\alpha_i-i$ elements of the form $q_{p+1-i,j}$: those with $j\in[p+1,p+\alpha_i-i]$. This is an order ideal because the tuple $(\alpha_1-1,\dots,\alpha_p-p)$ is weakly increasing.

\begin{example} \label{finite_ss_ideal}
The following is the Hasse diagram of the poset $(Q,\prec)$ in the case of $n=7$ and $p=3$ (arrows are oriented towards the lesser elements). Elements of the order ideal $\varphi_{\ss}((2,3,6))$ are colored cyan\footnote{Here and further we use light colors for the convenience of our readers using black and white printers.}. 
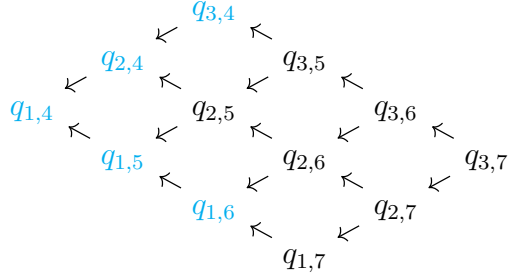
\begin{equation}
\begin{tikzcd}[row sep=.5mm,column sep=2.5mm]
&&{\color{cyan} q_{3,4}}\arrow[dl]\\
&{\color{cyan}q_{2,4}}\arrow[dl]&&  q_{3,5}\arrow[ul]\arrow[dl]\\
{\color{cyan}q_{1,4}}&&q_{2,5}\arrow[dl]\arrow[ul] &&   q_{3,6}\arrow[dl]\arrow[ul]\\
&{\color{cyan}q_{1,5}}\arrow[ul]&&q_{2,6}\arrow[dl]\arrow[ul] && q_{3,7}\arrow[dl]\arrow[ul]\\
&&{\color{cyan}q_{1,6}}\arrow[ul]&&q_{2,7}\arrow[dl]\arrow[ul]\\
&&&q_{1,7}\arrow[ul]
\end{tikzcd}
\end{equation}
\end{example}
The following fact is easily verified directly, the observation is essentially due to~\cite{KM}.
\begin{lemma}\label{ashibiring}
$\initial_{<_{\ss}}\cP$ is isomorphic to $A_{Q,\varnothing}(Q,\prec)$ as an $\bN_0$-graded algebra, the isomorphism is given by \[\initial_{<_{\ss}}D_I\mapsto (-1)^{n\choose 2}s\mathbf y^{v_{Q,\varnothing}(\varphi_{\ss}(I))}.\]
\end{lemma}
We see that the sagbi degeneration of $\Gr(p,n)$ obtained from Theorem~\ref{diagdegen} via Corollary~\ref{geomflatfam} is, in fact, the toric variety of the order polytope $\Pi_{Q,\varnothing}(Q,\prec)$ (which can alternatively be interpreted as the Gelfand--Tseltin polytope of a fundamental $\mathfrak{gl}_n$-weight, see~\cite{KM,ABS}).

The set $\cJ(Q)$ is naturally a poset with inclusion as the order relation. We also consider a partial order on the set $\mathcal I_{\ss}$. For tuples $I = (\alpha_1, \dots, \alpha_p)$ and $I' = (\alpha'_1, \dots, \alpha'_p)$ in $\mathcal I_{\ss}$ we write $I\preceq_{\ss} I'$ when $\alpha_k \leq \alpha'_k$ for all $k\in[1,p]$. Let $\prec_{\ss}$ denote the corresponding strict relation. The following is immediate from the  definitions.
\begin{proposition}\label{ssbijection}
$\varphi_{\ss}$ is an isomorphism between the posets $(\mathcal I_{\ss},\prec_{\ss})$ and $(\cJ(Q),\subset)$.
\end{proposition}

Recall that a Young tableau is \textit{semistandard} if the elements in each column increase strictly from top to bottom and the elements in each row increase non-strictly from left to right. Thus we have \[I_1\preceq_{\ss}\dots\preceq_{\ss} I_k\] if and only if the tableau with columns $I_1,\dots,I_k$ is semistandard. 

Now, Corollary~\ref{hibibasis} provides a basis in $A_{Q,\varnothing}(Q,\prec)$ parametrized by weakly increasing tuples in $(\cJ(Q),\subset)$. Such tuples are in bijection with semistandard Young tableaux with all columns of height $p$. Lemma~\ref{ashibiring} allows us to transfer this basis to $\initial_{<_{\ss}}\cP$. Finally, by applying Corollary~\ref{liftbasis} we lift the basis to $\cP$ and obtain the following classical fact.
\begin{theorem}\label{semistandardbasis}
The set of products $D_{I_1}\dots D_{I_k}$ for which the Young tableau with columns $I_1,\dots,I_k$ is semistandard is a basis in $\cP$.
\end{theorem}

\subsection{The Feigin--Fourier--Littelmann degeneration}\label{pbw}

We now describe a similar approach to a more recently constructed toric degeneration of the Grassmannian, first obtained in~\cite{FFL2}.
\begin{definition}
We say that a tuple $(\alpha_1,\dots,\alpha_p)$ of integers from $[1,n]$ is a \textit{PBW tuple} if
\begin{itemize}
\item the $\alpha_j$ are pairwise distinct,
\item for all $\alpha_j\le p$ we have $\alpha_j=j$ and
\item if $\alpha_j>\alpha_k>p$, then $j<k$.
\end{itemize}
\end{definition}
In other words, in a PBW tuple elements no greater than $p$ stand in positions equal to their values while other elements are placed in the remaining positions in decreasing order. In particular, each $p$-subset of $[1,n]$ can be uniquely ordered to form a PBW tuple. The terminology stems from~\cite{Fe} which considers PBW (Poincaré--Birkhoff--Witt) degenerations. We let $\mathcal I_{\pbw}$ denote the set of PBW tuples. 

Generalizing the above notation, for any tuple $I=(\alpha_1,\dots,\alpha_p)$ of integers from $[1,n]$ let $Z_I$ denote the $p\times p$ matrix whose $j$th column is equal to the $\alpha_j$th column of $Z$. In this case we also denote $D_I=\det Z_I$. Obviously, $D_I=\pm D_{I'}$ when $I$ and $I'$ are permutations of each other. Hence, the $D_I$ with $I\in\mathcal I_{\pbw}$ generate $\cP$. 

Next, let $<_{\pbw}$ denote the degree reverse lexicographic monomial order on $\bC[\{z_{i,j}\}]$ given by the following ordering of the variables:
\begin{equation}\label{pbwordering}
    \begin{matrix} 
z_{1,1} , z_{1,2} , \dots  , z_{1,n} ,\\
z_{2,2} , z_{2,3} , \dots  , z_{2,1} ,\\
\dots \\
z_{p,p} , z_{p,p+1} , \dots  , z_{p,p-1},
\end{matrix}
\end{equation}
i.e.\ first by row and then cyclically starting with $z_{i,i}$ within the $i$th row. It easily seen that for a PBW tuple $I$ one has \[\initial_{<_{\pbw}}D_I=z_{1,\alpha_1},\dots,z_{p,\alpha_p}.\]

\begin{theorem}\label{pbwdegen}
The polynomials $D_I$ with $I\in\mathcal I_{\pbw}$ form a sagbi basis of $\cP$ with respect to $<_{\pbw}$. 
\end{theorem}
\begin{proof}
Let $\mathcal A\subset\initial_{<_{\pbw}}\cP$ be the subalgebra generated by $\initial_{<_{\pbw}}D_I$ with $I\in\cI_{\pbw}$. We are to show that $\mathcal A=\initial_{<_{\pbw}}\cP$.

Now, as will be discussed below (see Lemma~\ref{ashibili}), the subalgebra generated by the monomials $\initial_{<_{\pbw}}D_I$ is isomorphic to the toric ring of the FFLV polytope of $p$th fundamental $\mathfrak{sl}_n$-weight. These polytopes are well-studied, the toric ring is discussed, for instance, in the proof of~\cite[Theorem 5.1]{FFFM} (in the greater generality of partial flag varieties). There it is shown that it has the same Hilbert series as the Pl\"ucker algebra. Corollary~\ref{samesize} then implies that $\mathcal A$ may not be a proper subalgebra of $\initial_{<_{\pbw}}\cP$.
\end{proof}

In this case $\initial_{<_{\pbw}}\cP$ can also be interpreted as a generalized Hibi ring. Indeed, let $I=(\alpha_1,\dots,\alpha_p)$ be a PBW tuple and let $(j_1<\dots<j_l)$ be all $j$ with $\alpha_j>p$. Then one sees that the set \[L=\{q_{j_1,\alpha_{j_1}},\dots,q_{j_l,\alpha_{j_l}}\}\] is an antichain in $(Q,\prec)$ and, moreover, every antichain is obtained in this manner. There exists a unique order ideal $J\subset Q$ with $\max_\prec J=L$ and the integer point $v_{\varnothing,Q}(J)\in\Pi_{\varnothing,Q}(Q,\prec)$ is the indicator function $\mathbf 1_L$. We set $\varphi_{\pbw}(I)=J$.

\begin{example} \label{finite_pbw_ideal}
In the below diagram the order ideal $\varphi_{\pbw}((6,2,4))$ for $n=7$ and $p=3$ is shown in cyan and red, with red elements ($q_{3,4}$ and $q_{1,6}$) lying in the corresponding antichain $L$. Note that this is the same ideal as in Example \ref{finite_ss_ideal}. 
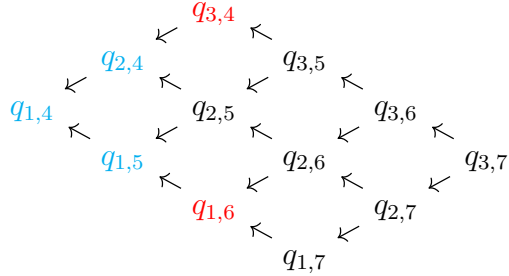
\begin{equation}
\begin{tikzcd}[row sep=.5mm,column sep=2.5mm]
&&\color{red}q_{3,4}\arrow[dl]\\
&\color{cyan}q_{2,4}\arrow[dl]&&  q_{3,5}\arrow[ul]\arrow[dl]\\
\color{cyan}q_{1,4}&&q_{2,5}\arrow[dl]\arrow[ul] &&   q_{3,6}\arrow[dl]\arrow[ul]\\
&\color{cyan}q_{1,5}\arrow[ul]&&q_{2,6}\arrow[dl]\arrow[ul] && q_{3,7}\arrow[dl]\arrow[ul]\\
&&\color{red}q_{1,6}\arrow[ul]&&q_{2,7}\arrow[dl]\arrow[ul]\\
&&&q_{1,7}\arrow[ul]
\end{tikzcd}
\end{equation}
\end{example}

The next lemma is straightforward from the definitions.
\begin{lemma}\label{ashibili}
$\initial_{<_{\pbw}}\cP$ is isomorphic to $A_{\varnothing,Q}(Q,\prec)$ as an $\bN_0$-graded algebra, the isomorphism is given by \[\initial_{<_{\pbw}}D_I\mapsto s\mathbf y^{v_{\varnothing,Q}(\varphi_{\pbw}(I))}\] for $I\in\mathcal I_{\pbw}$.
\end{lemma}
\begin{proof}
One can easily check that the inverse isomorphism is given by $y_{q_{i,j}}\mapsto z_{i,j}z_{i,i}^{-1}$ and $s\mapsto z_{1,1}\dots z_{p,p}$.
\end{proof}

We see that the sagbi degeneration $\Proj(\initial_{<_{\pbw}}\cP)$ is, in fact, the toric variety of the polytope $\Pi_{\varnothing,Q}(Q,\prec)$. This polytope can also be viewed as the Feigin--Fourier--Littelmann--Vinberg polytope of a fundamental weight, see~\cite{FFL1,ABS}.

We also define a partial order on $\mathcal I_{\pbw}$. For PBW tuples $I=(\alpha_1,\dots,\alpha_p)$ and $I'=(\alpha'_1,\dots,\alpha'_p)$ we set $I\preceq_{\pbw} I'$ when for every $k\in[1,p]$ there exists $l\ge k$ such that $\alpha'_l\ge \alpha_k$. Let $\prec_{\pbw}$ be the strict relation. The following fact is straightforward to check. 
\begin{proposition}[cf.\ {\cite[Lemma 6.4]{Makhlin}}]
$\varphi_{\pbw}$ is an isomorphism between the posets $(\mathcal I_{\pbw},\prec_{\pbw})$ and $(\mathcal J(Q,\prec),\subset)$.
\end{proposition}

\begin{definition}
Consider a rectangular Young tableau of height $p$ on the alphabet $[1,n]$ with columns $I_1,\dots,I_k$. We say that this tableau is \textit{PBW-semistandard} if every $I_j\in\mathcal I_{\pbw}$ and  \[I_1\preceq_{\pbw}\dots\preceq_{\pbw} I_k.\]
\end{definition}

\begin{remark}
It should be pointed out that the above definition differs from the more common definition originally given in~\cite{Fe} by a reversal of the order of columns (i.e.\ in the original definition they form a weakly decreasing tuple). This is done to keep the conventions natural and the two cases analogous to each other.
\end{remark}

Similarly to Theorem~\ref{semistandardbasis} we may now deduce the following fact.
\begin{theorem}
The set of products $D_{I_1}\dots D_{I_k}$ for which the Young tableau with columns $I_1,\dots,I_k$ is PBW-semistandard is a basis in $\cP$.
\end{theorem}


\section{The Sottile--Sturmfels degeneration}\label{SS}

The goal of this section is to present an approach to the results in~\cite{SS} similar to that of Subsections~\ref{ss} and~\ref{pbw}. In fact, the construction described here extends the construction in Subsection~\ref{ss}.

Let $n$ and $p$ be as above. Consider variables $z_{i,j}^{(k)}$ with $i\in [1,p]$, $j\in[1,n]$ and $k\in\mathbb \bN_0$. Let $z_{i,j}(t)$ denote the formal power series $\sum_k z_{i,j}^{(k)} t^k$ and let $Z(t)$ be the $p\times n$ matrix with elements $z_{i,j}(t)$. For a $p$-tuple $I=(\alpha_1,\dots,\alpha_p)$ of integers in $[1,n]$ we denote by $Z_I(t)$ the $p\times p$ matrix whose $j$th column is the $\alpha_j$th column of $Z(t)$. We set $D_I(t)=\det Z_I(t)$ and let $D_I^{(k)}$ be the coefficient of $t^k$ in $D_I(t)$.

\begin{definition}
The \textit{semi-infinite Pl\"ucker algebra} $\tilde\cP$ is the subalgebra in $\bC[\{z_{i,j}^{(k)}\}]$ generated by polynomials $D_I^{(k)}$ for all possible $I$ and $k$.
\end{definition}

We have an $\bN_0$-grading on $\tilde\cP$ equal to $1$ on the generators $D_I^{(k)}$. The scheme of infinite type $\Proj\tilde\cP$ is known as the \textit{quantum Grassmannian} (this is the terminology used in~\cite{SS}) or the \textit{semi-infinite Grassmannian} (this terminology originates in~\cite{FeFr} and is used in literature on affine Lie algebras and related subjects). Thus initial subalgebras of $\tilde\cP$ can be used to obtain flat degenerations of this variety via Corollary~\ref{geomflatfam}.

Let $\tilde\cI_{\ss}$ denote the set of all symbols of the form $I^{(k)}$ where $I$ is an \textbf{increasing} $p$-tuple in $[1,n]$ and $k\in\bN_0$. We have $D_I^{(k)}=\pm D_{I'}^{(k)}$ when $I'$ is a permutation of $I$, hence $\tilde\cP$ is generated by $D_I^{(k)}$ with $I^{(k)}\in\tilde\cI_{\ss}$. Let $<_{\ss}$ denote the degree reverse lexicographic monomial order on $\bC[\{z_{i,j}^{(k)}\}]$ given by ordering the $z_{i,j}^{(k)}$ lexicographically: first by $k$, then by $i$ and then by $j$. We remark that if one identifies $z_{i,j}^{(0)}$ with $z_{i,j}$, the defined monomial order extends the order $<_{\ss}$ considered in Subsection~\ref{ss}, hence the abuse of notation. The same will apply to several other notions introduced in this section and the next: $\varphi_{\ss}$, $\prec_{\ss}$, $<_{\pbw}$, $\varphi_{\pbw}$ and $\prec_{\pbw}$ extend the corresponding relations or maps in Subsections~\ref{ss} and~\ref{pbw}.
\begin{theorem}\label{sagbi}
The polynomials $D_I^{(k)}$ with $I^{(k)}\in\tilde\cI_{\ss}$ form a sagbi basis of $\tilde\cP$ with respect to $<_{\ss}$.
\end{theorem}
\begin{proof}
This is essentially~\cite[Theorem 1]{SS}. The only difference is that they consider a truncated finitely-generated version of the algebra $\tilde\cP$ while we pass to a direct limit. This is done as follows.

Choose $d\in\bN_0$ and let \[\pi_d:\bC[\{z_{i,j}^{(k)}\}]\to \bC[\{z_{i,j}^{(k)}\}_{k\le d}]\] be the projection taking every $z_{i,j}^{(k)}$ with $k>d$ to 0. Denote $\cP^d=\pi_d(\tilde\cP)$ and $D_I^{(k,d)}=\pi_d(D_I^{(k)})$. In~\cite[Theorem 1]{SS} it is shown that $D_I^{(k,d)}$ with $I^{(k)}\in\tilde\cI_{\ss}$ form a sagbi basis of $\cP^d$ with respect to $<_{\ss}$ which means that $\initial_{<_{\ss}}\cP^d$ is generated by the monomials $\initial_{<_{\ss}}D_I^{(k,d)}$. Moreover,~\cite[Lemma 3]{SS} shows that $\initial_{<_{\ss}}D_I^{(k,d)}$ does not depend on $d$ as long as $D_I^{(k,d)}\neq0$ (i.e.\ $k\le pd$). However, for every $I^{(k)}$ for large enough $d$ $D_I^{(k,d)}=D_I^{(k)}$, hence the subalgebra generated by the $\initial_{<_{\ss}}D_I^{(k)}$ contains all $\cP^d$ and coincides with $\tilde\cP$.
\end{proof}

The initial monomials are not hard to find.
\begin{proposition}[{\cite[Lemma 3]{SS}}]\label{initialterm}
For $I^{(k)}\in\tilde\cI_{\ss}$ let $k=lp+r$ with $l\in\bN_0$ and $r\in[0,p-1]$. Let $I=(\alpha_1,\dots,\alpha_p)$. Then \[\initial_{<_{\ss}}D_I^{(k)}=(-1)^{{n\choose2}+r(n-1)}z_{r+1,\alpha_p}^{(l)}z_{r+2,\alpha_{p-1}}^{(l)}\dots z_{p,\alpha_{r+1}}^{(l)}z_{1,\alpha_r}^{(l+1)}\dots z_{r,\alpha_1}^{(l+1)}.\]
\end{proposition}

We note that the above sign is not present in~\cite{SS} because of slightly different conventions but it is easily found as the sign of the permutation $(r,\dots,1,p,\dots,r+1)$. 

Similarly to the finite case we can interpret $\initial_{<_{\ss}}\tilde\cP$ as a generalized Hibi ring. 

\begin{definition}\label{precdef}
Let $(\tilde Q,\prec)$ be the poset consisting of elements $q_{i,j}$ with $i\ge 1$, $j\ge p+1$ and $j-i\in[0,n-1]$. We set $q_{i_1,j_1}\preceq\ q_{i_2,j_2}$ if at least one of the following holds:
\begin{enumerate}[label=(\roman*)]
\item $i_1\le i_2$ and $j_1\le j_2$,
\item $i_1+p\le i_2$ or
\item $j_1+n-p\le j_2$.
\end{enumerate}
We write $\prec$ for the corresponding strict relation.
\end{definition}
The first thing to note is that $(\tilde Q,\prec)$ contains $(Q,\prec)$ as a subposet because neither of (ii) and (iii) is possible when both elements lie in $Q$. One also easily checks that the defined relation is antisymmetric 
and transitive.

\begin{remark} \label{precex}
There is a convenient way to visualize this poset: one can find a part of the Hasse diagram of $(\tilde Q,\prec)$ in the case of $n=7$ and $p=3$ in the diagram \eqref{visualization} (the cyan color can be ignored for the moment). The arrows corresponding to relations of types (ii) and (iii) are gray and dashed for better visibility. 

This poset also has the following interpretation. Consider the set $H$ consisting of all elements of the form $q_{i+mp,j-m(n-p)}$ where $q_{i,j}\in\tilde Q$ and $m\in\bZ$. This set is equipped by an order $\prec'$ such that $q_{i_1,j_1}\preceq'q_{i_2,j_2}$ if and only if $i_1\le i_2$ and $j_1\le j_2$. Here is a fragment of the Hasse diagram of $(H,\prec')$ for $n=5$ and $p=2$:
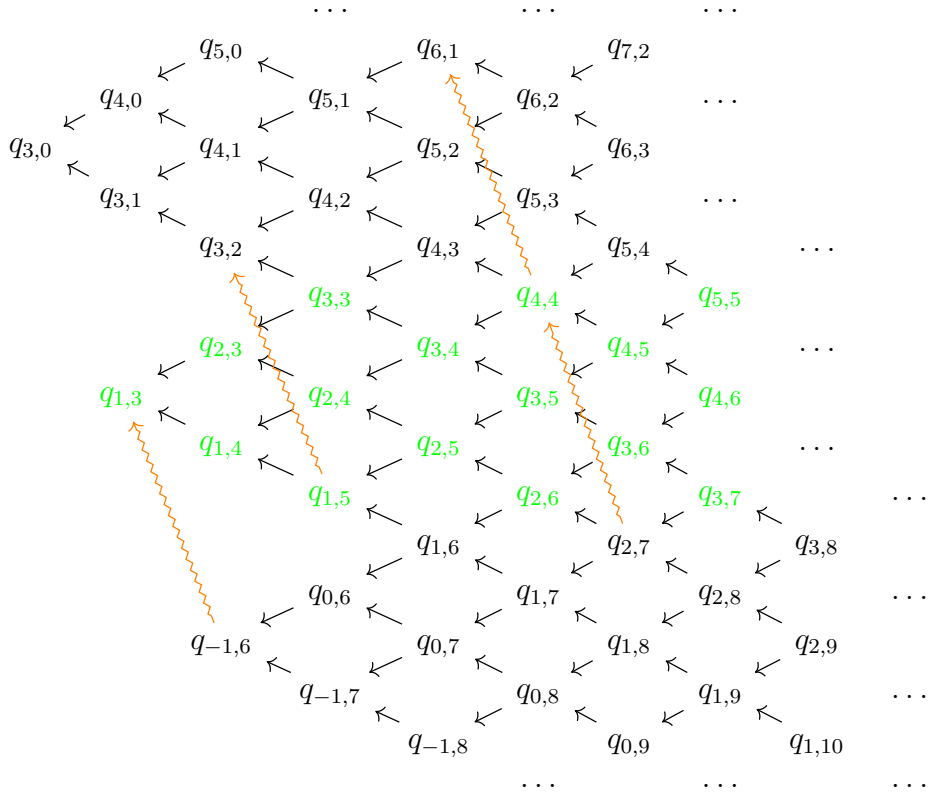
\begin{center}
\begin{tikzcd}[row sep=.5mm,column sep=2.5mm]
&&&\dots&&\dots&&\dots\\
&&q_{5,0}\arrow[dl]&& q_{6,1}\arrow[dl] &&q_{7,2}\arrow[dl]\\
&q_{4,0}\arrow[dl]&& q_{5,1}\arrow[dl]\arrow[ul] && q_{6,2}\arrow[dl]\arrow[ul] && \dots\\
q_{3,0}&&  q_{4,1}\arrow[ul]\arrow[dl] && q_{5,2}\arrow[dl]\arrow[ul] && q_{6,3}\arrow[dl]\arrow[ul]\\
&q_{3,1}\arrow[ul] &&    q_{4,2}\arrow[dl]\arrow[ul] && q_{5,3}\arrow[dl]\arrow[ul]  && \dots\\
&&q_{3,2}\arrow[ul] && q_{4,3}\arrow[ul]\arrow[dl] && q_{5,4}\arrow[ul]\arrow[dl]&&\dots\\
&&&\color{green}q_{3,3}\arrow[dl]\arrow[ul]&& \color{green}q_{4,4}\arrow[dl]\arrow[ul]\arrow[uuuuul,squiggly,orange] &&\color{green}q_{5,5}\arrow[dl]\arrow[ul]\\
&&\color{green}q_{2,3}\arrow[dl]&& \color{green}q_{3,4}\arrow[dl]\arrow[ul] && \color{green}q_{4,5}\arrow[dl]\arrow[ul] && \dots\\
&\color{green}q_{1,3}&&  \color{green}q_{2,4}\arrow[ul]\arrow[dl] && \color{green}q_{3,5}\arrow[dl]\arrow[ul] && \color{green}q_{4,6}\arrow[dl]\arrow[ul]\\
&&\color{green}q_{1,4}\arrow[ul] &&    \color{green}q_{2,5}\arrow[dl]\arrow[ul] && \color{green}q_{3,6}\arrow[dl]\arrow[ul]  && \dots\\
&&&\color{green}q_{1,5}\arrow[ul]\arrow[uuuuul,squiggly,orange] && \color{green}q_{2,6}\arrow[ul]\arrow[dl] && \color{green}q_{3,7}\arrow[ul]\arrow[dl]&&\dots\\
&&&&q_{1,6}\arrow[dl]\arrow[ul]&& q_{2,7}\arrow[ul]\arrow[dl]\arrow[uuuuul,squiggly,orange] &&q_{3,8}\arrow[dl]\arrow[ul]\\
&&&q_{0,6}\arrow[dl]&& q_{1,7}\arrow[dl]\arrow[ul] && q_{2,8}\arrow[dl]\arrow[ul] && \dots\\
&&q_{-1,6}\arrow[uuuuul,squiggly,orange]&&  q_{0,7}\arrow[ul]\arrow[dl] && q_{1,8}\arrow[dl]\arrow[ul] && q_{2,9}\arrow[dl]\arrow[ul]\\
&&&q_{-1,7}\arrow[ul] &&    q_{0,8}\arrow[dl]\arrow[ul] && q_{1,9}\arrow[dl]\arrow[ul]  && \dots\\
&&&&q_{-1,8}\arrow[ul] && q_{0,9}\arrow[ul] && q_{1,10}\arrow[ul]\\
&&&&&\dots&&\dots&&\dots\\
\end{tikzcd}
\end{center}
Here the orange squiggly arrows are not part of the Hasse diagram but visualize the translation $\varepsilon:q_{i,j}\mapsto q_{i+p,j-(n-p)}$, the subset $\tilde Q\subset H$ is highlighted in green. Note that $H$ is preserved by $\varepsilon$ and that $\tilde Q$ is a fundamental region for the group generated by $\varepsilon$. Moreover, it is easily checked that the quotient poset $(H,\prec')/\langle\varepsilon\rangle$ is well-defined and naturally isomorphic to $(\tilde Q,\prec)$. In particular, since we have embedded $\tilde Q$ into the quotient of the plane modulo a translation, it can be thought of as a poset of points in the cylinder.
\end{remark}

We denote $\cJ(\tilde Q)=\cJ(\tilde Q,\prec)$. Our next goal is to define a bijection between $\tilde\cI_{\ss}$ and $\cJ(\tilde Q)$. 
\begin{definition}
For $I^{(k)}\in\tilde\cI_{\ss}$ with $I=(\alpha_1,\dots,\alpha_p)$ let $\varphi_{\ss}(I^{(k)})\subset\tilde Q$ be the subset containing all $q_{i,j}$ with $i\le k$ and also for every $i\in[k+1,k+p]$ containing those $q_{i,j}$ for which $j\le i+\alpha_{p+1-(i-k)}-1$.
\end{definition}
In terms of the visualization in~\eqref{visualization}, for $i\in[k+1,k+p]$ the defined set contains the top $\alpha_{p+1-(i-k)}$ elements of the form $q_{i,\bullet}$ when $i>p$ and the top $\alpha_{p+1-(i-k)}-(p+1-i)$ such elements when $i\le p$. For $i\le k$ it contains all elements of the form $q_{i,\bullet}$. 
\begin{example}\label{ss_ideal}
Let us depict the order ideal corresponding to $(2,3,6)^{(2)}$ in case of $n=7$ and $p=3$ (elements of the ideal are colored cyan). 
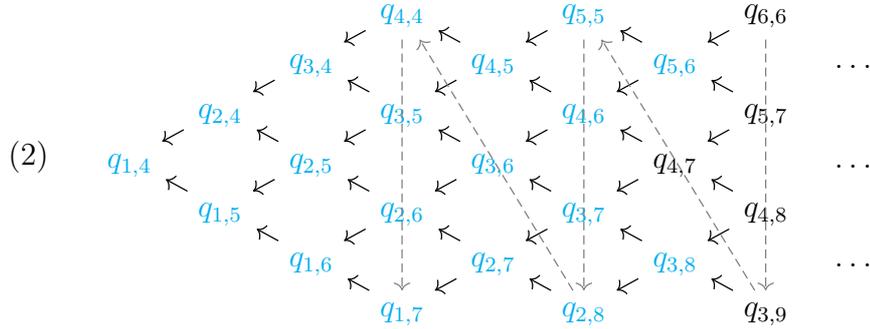
\begin{equation}\label{visualization}
\begin{tikzcd}[row sep=.5mm,column sep=2.5mm]
&&&\color{cyan}q_{4,4}\arrow[dl]\arrow[dddddd, dashed, gray] && \color{cyan}q_{5,5}\arrow[dl] \arrow[dddddd, dashed, gray]&&q_{6,6}\arrow[dl]\arrow[dddddd, dashed, gray]\\
&&\color{cyan}q_{3,4}\arrow[dl]&& \color{cyan}q_{4,5}\arrow[dl]\arrow[ul] && \color{cyan}q_{5,6}\arrow[dl]\arrow[ul] && \dots\\
&\color{cyan}q_{2,4}\arrow[dl]&&  \color{cyan}q_{3,5}\arrow[ul]\arrow[dl] && \color{cyan}q_{4,6}\arrow[dl]\arrow[ul] && q_{5,7}\arrow[dl]\arrow[ul]\\
\color{cyan}q_{1,4}&&\color{cyan}q_{2,5}\arrow[dl]\arrow[ul] && \color{cyan}   q_{3,6}\arrow[dl]\arrow[ul] && q_{4,7}\arrow[dl]\arrow[ul]  && \dots\\
&\color{cyan}q_{1,5}\arrow[ul]&&\color{cyan}q_{2,6}\arrow[dl]\arrow[ul] && \color{cyan}q_{3,7}\arrow[dl]\arrow[ul] && q_{4,8}\arrow[dl]\arrow[ul]\\
&&\color{cyan}q_{1,6}\arrow[ul]&&\color{cyan}q_{2,7}\arrow[dl]\arrow[ul] && \color{cyan}q_{3,8}\arrow[dl]\arrow[ul]  && \dots\\
&&&\color{cyan}q_{1,7}\arrow[ul] && \color{cyan}q_{2,8}\arrow[ul]\arrow[lluuuuuu, dashed, gray] && q_{3,9}\arrow[ul]\arrow[lluuuuuu, dashed, gray]
\end{tikzcd}
\end{equation}
It is advisable to compare this diagram to the one in Example  \ref{finite_ss_ideal}, one can see that the diagram of the ideal $\varphi_{\ss}((2,3,6)^{(2)})$ can be obtained from that of $\varphi_{\ss}((2,3,6)^{(0)})$ by ``shifting two steps to the right''. This is, in fact, a general rule, cf.\ proof of Proposition~\ref{infssbijection}.
\end{example}

\begin{proposition}\label{infssbijection}
$\varphi_{\ss}$ is a bijection from $\tilde\cI_{\ss}$ to $\cJ(\tilde Q)$.
\end{proposition}
\begin{proof}
Let us first describe the finite order ideals in $(\tilde Q,\prec)$. For $J\in\cJ(\tilde Q)$ consider the largest $k$ such that $q_{k+p,k+p}\in J$, if no such $k$ exists, set $k=0$. From Definition~\ref{precdef} we deduce that:
\begin{itemize}
\item $q_{i,j}\in J$ when $j\le k+p$ due to (i) (since $i\le j$),
\item $q_{i,j}\in J$ when $i\le k$ due to (ii),
\item $q_{i,j}\notin J$ when $i\ge k+p+1$ due to $q_{k+p+1,k+p+1}\notin J$ and (i),
\item $q_{i,j}\notin J$ when $j\ge k+n+1$ due to $q_{k+p+1,k+p+1}\notin J$ and (iii).
\end{itemize}

In terms of our visualization of $\tilde Q$ this can be summarized as follows. Let $Q^{(k)}$ denote the ``$p\times(n-p)$ rectangle'' consisting of $q_{i,j}$ with $i\in[k+1,k+p]$ and $j\in[k+p+1,k+n]$ (in particular, $Q^{(0)}=Q$). Then $J$ contains everything ``to the left'' of $Q^{(k)}$ and contains nothing ``to the right'' of $Q^{(k)}$, thus being determined by its intersection with $Q^{(k)}$. However, the poset $(Q^{(k)},\prec)$ is isomorphic to $(Q,\prec)$, in particular, it has $n\choose p$ order ideals. We conclude that these are in bijection with those $J\in\cJ(\tilde Q)$ which contain $q_{k+p,k+p}$ but not $q_{k+p+1,k+p+1}$.

Now we note that $\varphi_{\ss}(I^{(k)})$ consists of all $q_{i,j}$ with $i\le k$ or $j\le k+p$ disjointly united with $K\subset Q^{(k)}$ where $K$ is obtained from $\varphi_{\ss}(I)\subset Q$ by replacing every $q_{i,j}$ with $q_{i+k,j+k}$ (``shifting $k$ steps to the right''). The statement now follows from Proposition~\ref{ssbijection}.
\end{proof}

\begin{lemma}\label{infssashibiring}
$\initial_{<_{\ss}}\tilde\cP$ is isomorphic to $A_{\tilde Q,\varnothing}(\tilde Q,\prec)$ as an $\bN_0$-graded algebra, the isomorphism is given by \[\initial_{<_{\ss}}D_I^{(k)}\mapsto (-1)^{{n\choose2}+r(n-1)}s\mathbf y^{v_{\tilde Q,\varnothing}(\varphi_{\ss}(I^{(k)}))}\]
where $k\equiv r\ (\mod\ p)$ with $r\in[0,p-1]$.
\end{lemma}

This lemma is easily deduced from the results in~\cite{SS}, the proof will be given after we show that $\varphi_{\ss}$ is not only a bijection but a poset isomorphism. By applying Corollary~\ref{geomflatfam}, we can draw the following conclusion from Theorem~\ref{sagbi} and Lemma~\ref{infssashibiring}.
\begin{corollary}\label{ssdegen}
The toric variety of the order polytope $\Pi_{\tilde Q,\varnothing}(\tilde Q,\prec)$ is a flat degeneration of the semi-infinite Grassmannian $\Proj\tilde P$.
\end{corollary}

Now to the poset structure on $\tilde\cI_{\ss}$.
\begin{definition}\label{precssdef}
For $I^{(k)},I'^{(k')}\in\tilde\cI_{\ss}$ with $I=(\alpha_1,\dots,\alpha_p)$ and $I'=(\alpha'_1,\dots,\alpha'_p)$ set $I^{(k)}\preceq_{\ss}I'^{(k')}$ if $k\le k'$ and for any $i,i'\in[1,p]$ with $i'-i=k'-k$ we have $\alpha_i\le\alpha'_{i'}$. Let $\prec_{\ss}$ be the corresponding strict relation.
\end{definition}

\begin{remark}
The poset $(\tilde\cI_{\ss},\prec_{\ss})$ is seen to be a distributive lattice. It is the obvious direct limit $q\to\infty$ of the posets $\mathcal C_q$ in~\cite{Sot} or lattices $\mathcal C_{\bullet,\bullet}^q$ in~\cite{SS}. It is also the lattice $\tilde{\mathcal L}$ appearing in the introduction.
\end{remark}

\begin{proposition}\label{posetisom}
$\varphi_{\ss}$ is an isomorphism between the posets $(\tilde\cI_{\ss},\prec_{\ss})$ and $(\cJ(\tilde Q),\subset)$.
\end{proposition}
\begin{proof}
Consider $I^{(k)},I'^{(k')}\in\tilde\cI_{\ss}$ with $I=(\alpha_1,\dots,\alpha_p)$, $I'=(\alpha'_1,\dots,\alpha'_p)$ and $k\le k'$. We are to check that $I^{(k)}\preceq_{\ss}I'^{(k')}$ if and only if for every $i\in\bN_0$ the set $\varphi_{\ss}(I^{(k)})$ contains no more elements of the form $q_{i,\bullet}$ than $\varphi_{\ss}(I'^{(k')})$ does. Note that we need not consider $i\le k'$ and $i>k+p$, since in the former case $\varphi_{\ss}(I'^{(k')})$ contains all such elements while in the latter case $\varphi_{\ss}(I^{(k)})$ contains no such elements. For $i\in[k'+1,k+p]$ (if they exist) the condition amounts to $\alpha_{p+1-(i-k)}\le\alpha'_{p+1-(i-k')}$ by the definition of $\varphi_{\ss}$. Setting $l=p+1-(i-k)$ we now see that $\varphi_{\ss}(I^{(k)})\subset\varphi_{\ss}(I'^{(k')})$ if and only if for all $l\in[1,p-(k'-k)]$ we have $\alpha_l\le\alpha_{l+k'-k}$. This is evidently equivalent to the condition in Definition~\ref{precssdef}.
\end{proof}

\begin{proof}[Proof of Lemma~\ref{infssashibiring}]
The claim is obtained from~\cite{SS} by translating from their language of distributive lattices to our language of posets. For a distributive lattice $(\cL,\land,\lor)$ consider the polynomial ring $\bC[\cL]$ in variables $X_a$ with $a\in\cL$. The ideal $H(\cL)\subset\bC[\cL]$ generated by all binomials of the form \[X_aX_b-X_{a\land b}X_{a\lor b}\] is known as the \textit{Hibi ideal} of the lattice $\cL$ and $\bC[\cL]/H(\cL)$ is known as its \textit{Hibi ring}. This overlaps with the terminology introduced in Section~\ref{polytopes} because of the following standard fact commonly attributed to~\cite{H}. Given a poset $(P,\ll)$ for which we have an isomorphism $\varphi:\cL\to\cJ(P,\ll)$ of distributive lattices (cf.\ \textit{Birkhoff's representation theorem}~\cite{wiki}), we have an isomorphism from $\bC[\cL]/H(\cL)$ to  $A_{P,\varnothing}(P,\ll)$ which maps the class of $X_a$ to $s\mathbf y^{v_{P,\varnothing}(\varphi(a))}$ (i.e.\ the zero set of the Hibi ideal is the toric variety of the order polytope).

As mentioned above, the poset $(\cI_{\ss},\prec_{\ss})$ is a distributive lattice. \cite[Theorem 2 and Proposition 4]{SS} show that there is an isomorphism from $\initial_{<_{\ss}}\tilde\cP$ to $\bC[\cI_{\ss}]/H(\cI_{\ss})$ which maps  $\initial_{<_{\ss}}D_I^{(k)}$ to the class of $X_{I^{(k)}}$. Combining this with the above standard fact we obtain the desired isomorphism.
\end{proof}

A way of visualizing the order $\prec_{\ss}$ is via semi-infinite semistandard tableaux. Consider the infinite square grid in the first quadrant consisting of boxes $(i,j)$ with $i,j\ge 1$ where $i$ (resp.\ $j$) is the coordinate along  the horizontal (resp.\ vertical) axis.
\begin{definition}\label{tableaudef}
For integers $k_1\le\dots\le k_m$ in $\bN_0$ a \textit{semi-infinite tableau} of \textit{shift} $(k_1,\dots,k_m)$ is the set of boxes $(i,k_i+j)$ with $i\in[1,m]$, $j\in[1,p]$ together with an integer written in each box (its \textit{content}). A semi-infinite tableau is \textit{semistandard} if the contents increase strictly from top to bottom in every column and non-strictly from left to right in every row.
\end{definition}
\begin{remark}
This notion is very similar to that of \textit{semi-infinite Young tableaux of shape $\omega_p$} (the $p$th fundamental weight) found in \cite{Ishii}. 
\end{remark}
\begin{example}
The following is an example of a semistandard semi-infinite tableau of shift $(0,1,3)$. The first $k_i$ boxes in each column are depicted in yellow since they are not part of the tableau.
{
\begin{center}
\begin{ytableau}
       \none & \none & 2 \\
       \none & \none & 5 \\
       \none & 1 & 7 \\
       1 & 2 & *(yellow) \\
       3 & 3 & *(yellow) \\
       5 & *(yellow) & *(yellow) 
\end{ytableau}
\end{center}
}
\end{example}

The definition above is a minor variation of the original definition in~\cite{SS} where these objects are viewed as (semi)standard skew Young tableaux. Consider a semi-infinite tableau of shift $(k_1,\dots,k_m)$ with the contents in its $i$th column forming a $p$-tuple $I_i$ when read from top to bottom. It is evident from the definitions that this semi-infinite tableau is semistandard if and only if all $I_i^{(k_i)}\in\tilde\cI_{\ss}$ and \[I_1^{(k_1)}\preceq_{\ss}\dots\preceq_{\ss}I_m^{(k_m)}.\] Similarly to Subsections~\ref{ss} and~\ref{pbw}, by combining Lemma~\ref{infssashibiring} with Corollary~\ref{hibibasis}, Proposition~\ref{liftbasis} and Proposition~\ref{posetisom} we obtain the following.
\begin{theorem}[implicit in~\cite{SS}]\label{infssbasis}
For a product $D_{I_1}^{(k_1)}\dots D_{I_m}^{(k_m)}$ consider the semi-infinite tableau of shift $(k_1,\dots,k_m)$ for which the contents in its $i$th column form the tuple $I_i$ when read from top to bottom. The set of products for which this tableau is semistandard is a basis in the semi-infinite Pl\"ucker algebra $\tilde\cP$.
\end{theorem}

\begin{remark}
To be abundantly clear we point out that the main results presented in this section (Theorems~\ref{sagbi} and~\ref{infssbasis}) are due to~\cite{SS}. However, the poset $(\tilde Q,\prec)$ is not considered in~\cite{SS} and, instead, combinatorially the focus is on the poset $(\tilde I_{\ss},\prec_{\ss})$ and its distributive lattice structure. In particular, it is mentioned that the sagbi degeneration $\Proj(\initial_{<_{\ss}}\tilde\cP)$ is toric but the corresponding polytope is not discussed. For us the poset $(\tilde Q,\prec)$ and its interpolating polytopes are instrumental in constructing a new degeneration of the semi-infinite Grassmannian in the next section. 

It is also worth mentioning that $(\tilde Q,\prec)$ is uniquely determined by $\cJ(\tilde Q)=(\tilde I_{\ss},\prec_{\ss})$ as the subposet of join-irreducible elements in the lattice $(\tilde I_{\ss},\prec_{\ss})$ (cf.\ \cite{wiki}). Nonetheless, identifying the join-irreducibles is not always the easiest task and the poset $(\tilde Q,\prec)$ was actually obtained ad hoc based on the finite case.
\end{remark}

\section{A new degeneration of the semi-infinite Grassmannian}\label{new}

We proceed to construct another toric degeneration of the semi-infinite Grassmannian $\Proj\tilde P$ in a fashion similar to the previous sections. The construction below extends the construction in Subsection~\ref{pbw}.

\begin{definition}\label{infpbwmonord}
Let $<_{\pbw}$ be the degree reverse lexicographic monomial order on $\bC[\{z_{i,j}^{(k)}\}]$ given by ordering the $z_{i,j}^{(k)}$ first by $k$ increasingly, then by $i$ increasingly and within a given $k$ and $i$ cyclically starting with $z_{i,kp+i\ \mod\ n}^{(k)}$, i.e.\ as \[z_{i,kp+i\ \mod\ n}^{(k)},z_{i,(kp+i\ \mod\ n)+1}^{(k)},\dots,z_{i,n}^{(k)},z_{i,1}^{(k)},\dots,z_{i,(kp+i\ \mod\ n)-1}^{(k)}.\] Here and further we understand $a\ \mod\ n$ to be an integer in $[1,n]$, we also understand the operation $\mod$ to have the lowest precedence, it being performed after addition and subtraction.
\end{definition}

In particular, one sees that the variables $z_{i,j}^{(0)}$ are ordered in the same way as the $z_{i,j}$ in~\eqref{pbwordering} (page \pageref{pbwordering}). However, among the $z_{1,\bullet}^{(1)}$ the variable $z_{1,p+1}^{(1)}$ comes first and among the $z_{2,\bullet}^{(1)}$ the first one is $z_{2,p+2\ \mod\ n}^{(1)}$.

\begin{definition}
Let $\tilde\cI_{\pbw}$ be the set of those $I^{(k)}$ with $I=(\alpha_1,\dots,\alpha_p)$ a $p$-tuple in $[1,n]$ and $k\in\bN_0$ for which \[(\alpha_1-k\;\;\;\mod\ n,\dots,\alpha_p-k\;\;\;\mod\ n)\in\cI_{\pbw}.\]
\end{definition}

\begin{example} \label{pbw_columns}
For $p=4$ and $n=9$ we have
\begin{itemize}
    \item $(1,2,6,5)^{(0)} \in \tilde\cI_{\pbw}$ because $(1,2,6,5) \in \cI_{\pbw}$,
    \item $(3,8,7,6)^{(2)} \in \tilde\cI_{\pbw}$ because $(1,6,5,4) \in \cI_{\pbw}$,
    \item $(5,1,7,9)^{(4)} \in \tilde\cI_{\pbw}$ because $(1,6,3,5) \in \cI_{\pbw}$. Note that in this case the ``mod $n$'' operation is essential.
\end{itemize}

\end{example}

It is evident that for every $k\in\bN_0$ and $p$-tuple $I$ in $[1,n]$ there exists a unique $I'^{(k)}\in\tilde\cI_{\pbw}$ such that $I'$ is a permutation of $I$. Therefore, $\tilde\cP$ is generated by the $D_I^{(k)}$ with $I^{(k)}\in\tilde\cI_{\pbw}$. One of the main results to be proved in this section is as follows.
\begin{theorem}\label{main1}
The polynomials $D_I^{(k)}$ with $I^{(k)}\in\tilde\cI_{\pbw}$ form a sagbi basis of $\tilde\cP$ with respect to $<_{\pbw}$.
\end{theorem}

We start by describing the initial terms.
\begin{proposition}
For $I^{(k)}\in\tilde\cI_{\pbw}$ with $I=(\alpha_1,\dots,\alpha_p)$ and $k=lp+r$ where $l\in\bN_0$ and $r\in[0,p-1]$ we have \[\initial_{<_{\pbw}}D_I^{(k)}=(-1)^{r(n-1)}z_{1,\alpha_{p-r+1}}^{(l+1)}\dots z_{r,\alpha_p}^{(l+1)}z_{r+1,\alpha_1}^{(l)}\dots z_{p,\alpha_{p-r}}^{(l)}.\]
\end{proposition}
\begin{proof}
The fact that we first order the $z_{i,j}^{(k)}$ by $k$ already means that $\initial_{<_{\pbw}}D_I^{(k)}$ must contain $p-r$ variables of the form $z_{\bullet,\bullet}^{(l)}$ and $r$ variables of the form $z_{\bullet,\bullet}^{(l+1)}$. Furthermore, since we then order by $i$, the initial term must (up to sign) have the form 
\begin{equation}\label{infpbwinitialterm}
z_{r,\alpha_{b_r}}^{(l+1)}\dots z_{1,\alpha_{b_1}}^{(l+1)}z_{p,\alpha_{b_p}}^{(l)}\dots z_{r+1,\alpha_{b_{r+1}}}^{(l)} 
\end{equation}
for some permutation $(b_1,\dots,b_p)$ of $[1,p]$. The $b_i$ are determined by passing from left to right in~\eqref{infpbwinitialterm} and choosing each $b_i$ so that it does not occur previously and the resulting $z_{i,b_i}^{(m)}$ ($m\in\{l,l+1\}$) comes as early as possible in our ordering.

For $i\in[1,p]$ let $c_i=\alpha_{b_i}-k\ \ \mod\ n$. Note that among the $z_{r,\bullet}^{(l+1)}$ the variable $z_{r,k+p\ \mod\ n}^{(l+1)}$ comes first while the remaining $z_{r,j}^{(l+1)}$ are ordered increasingly by $j-(k+p)\ \ \mod\ n$. Therefore, $b_r$ is chosen so that $c_r=p$ and, if this is not possible, so that $c_r-p\ \ \mod\ n$ is as small as possible. We see that, in view of $I^{(k)}\in\tilde\cI_{\pbw}$, we must have $b_r=p$. Similarly, $b_{r-1}$ is chosen so that $c_{r-1}=p-1$ and, if this is not possible, so that $c_{r-1}-(p-1)\ \ \mod\ n$ is as small as possible. Again, $I^{(k)}\in\tilde\cI_{\pbw}$ implies $b_{r-1}=p-1$. The same logic provides $b_i=p-r+i$ for $i=r-2,\dots,1$ and then $b_i=i-r$ for $i=p,\dots,r+1$. Finally, $(-1)^{r(n-1)}$ is the sign of the permutation $(b_1,\dots,b_p)$.
\end{proof}

We define a bijection between $\tilde\cI_{\pbw}$ and $\cJ(\tilde Q)$.
\begin{definition}\label{infpbwphi}
For $I^{(k)}\in\tilde\cI_{\pbw}$ with $I=(\alpha_1,\dots,\alpha_p)$ let $\varphi_{\pbw}(I^{(k)})\in\cJ(\tilde Q)$ be the minimal order ideal containing all elements $q_{k+i,k+(\alpha_i-k\ \mod\ n)}$ such that $i\in[1,p]$ and $k+(\alpha_i-k\ \ \mod\ n)>p$.
\end{definition}
\begin{example}\label{pbw_ideal}
Let us depict the order ideal corresponding to $(1,4,6)^{(2)} \in\tilde\cI_{\pbw}$ in case of $n=7$ and $p=3$. Elements of the ideal are colored cyan or red, the red ones ($q_{4,4}$, $q_{5,6}$ and $q_{3,8}$) are those that appear in Definition~\ref{infpbwphi} as $q_{k+i,k+(\alpha_i-k\ \mod\ n)}$.
\begin{center}

\begin{tikzcd}[row sep=.5mm,column sep=2.5mm]
&&&\color{red}q_{4,4}\arrow[dl]\arrow[dddddd, dashed, gray] && \color{cyan}q_{5,5}\arrow[dl] \arrow[dddddd, dashed, gray]&&q_{6,6}\arrow[dl]\arrow[dddddd, dashed, gray]\\

&&\color{cyan}q_{3,4}\arrow[dl]&& \color{cyan}q_{4,5}\arrow[dl]\arrow[ul] && \color{red}q_{5,6}\arrow[dl]\arrow[ul] && \dots\\

&\color{cyan}q_{2,4}\arrow[dl]&&  \color{cyan}q_{3,5}\arrow[ul]\arrow[dl] && \color{cyan}q_{4,6}\arrow[dl]\arrow[ul] && q_{5,7}\arrow[dl]\arrow[ul]\\

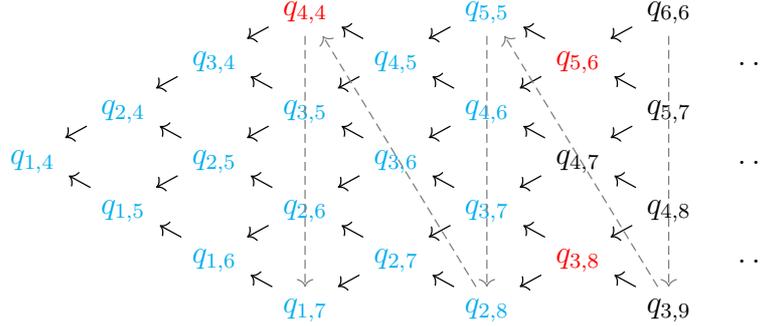
\color{cyan}q_{1,4}&&\color{cyan}q_{2,5}\arrow[dl]\arrow[ul] && \color{cyan}   q_{3,6}\arrow[dl]\arrow[ul] && q_{4,7}\arrow[dl]\arrow[ul]  && \dots\\

&\color{cyan}q_{1,5}\arrow[ul]&&\color{cyan}q_{2,6}\arrow[dl]\arrow[ul] && \color{cyan}q_{3,7}\arrow[dl]\arrow[ul] && q_{4,8}\arrow[dl]\arrow[ul]\\

&&\color{cyan}q_{1,6}\arrow[ul]&&\color{cyan}q_{2,7}\arrow[dl]\arrow[ul] && \color{red}q_{3,8}\arrow[dl]\arrow[ul]  && \dots\\

&&&\color{cyan}q_{1,7}\arrow[ul] && \color{cyan}q_{2,8}\arrow[ul]\arrow[lluuuuuu, dashed, gray] && q_{3,9}\arrow[ul]\arrow[lluuuuuu, dashed, gray]
\end{tikzcd}
\end{center}
Note that the diagram of $\varphi_{\pbw}((1,4,6)^{(2)})$ is obtained from the one of $\varphi_{\pbw}((6,2,4)^{(0)})$ in Example~\ref{finite_pbw_ideal} by ``shifting two steps to the right''. Here $(6,2,4)$ is obtained from $(1,4,6)$ by subtracting (modulo $n$) the superscript 2.
\end{example}

\begin{proposition}\label{infpbwbijection}
$\varphi_{\pbw}$ is a bijection from $\tilde\cI_{\pbw}$ to $\cJ(\tilde Q)$.
\end{proposition}
\begin{proof}
For $I^{(k)}\in\tilde\cI_{\pbw}$ consider \[I'=(\beta_1,\dots,\beta_p)=(\alpha_1-k\;\;\;\mod\ n,\dots,\alpha_p-k\;\;\;\mod\ n)\in\cI_{\pbw}.\] As discussed in Subsection~\ref{pbw}, the elements $q_{i,\beta_i}$ with $\beta_i>p$ form an antichain $L\subset(Q,\prec)$ and $\varphi_{\pbw}(I')$ is the order ideal generated by $L$. Let $L^{(k)}$ be the antichain in $(\tilde Q,\prec)$ obtained from $L$ by replacing every $q_{i,j}$ with $q_{i+k,j+k}$, i.e.\ by ``shifting $k$ steps to the right'' in terms of the visualization in~\eqref{visualization}. 

The elements $q_{k+i,k+(\alpha_i-k\ \mod\ n)}$ in Definition~\ref{infpbwphi} are precisely the elements of $L^{(k)}$ together with $q_{k+i,k+i}$ for all $i=\beta_i$. Since $q_{k+p,k+\beta_p}\succeq q_{k+p,k+p}$, we necessarily have $q_{k+p,k+p}\in\varphi_{\pbw}(I^{(k)})$. Hence, $\varphi_{\pbw}(I^{(k)})$ is generated by $L^{(k)}$ and (when $k\ge 1$) the element $q_{k+p,k+p}$. In other words, it consists of the order ideal $K$ in $(Q^{(k)},\prec)$ (see proof of Proposition~\ref{infssbijection}) with $\max_\prec K=L^{(k)}$ and of all $q_{i,j}$ with $i\le k$ or $j\le k+p$ (everything ``to the left'' of $Q^{(k)}$). However, we have seen that all ideals in $(\tilde Q,\prec)$ have such form.
\end{proof}

Let $O\subset\tilde Q$ consist of all $q_{i,i}$ and $C=\tilde Q\backslash O$. Let $\mathcal A\subset\bC[\{z_{i,j}^{(k)}\}]$ denote the subalgebra generated by the monomials $\initial_{<_{\pbw}}D_I^{(k)}$ with $I^{(k)}\in\tilde\cI_{\pbw}$. The key to proving Theorem~\ref{main1} is the following fact.
\begin{lemma}\label{infpbwashibiring}
$\mathcal A$ is isomorphic to $A_{O,C}(\tilde Q,\prec)$ as an $\bN_0$-graded algebra, the isomorphism is given by \[\initial_{<_{\pbw}}D_I^{(k)}\mapsto (-1)^{r(n-1)}s\mathbf y^{v_{O,C}(\varphi_{\pbw}(I^{(k)}))}\]
where $k\equiv r\ (\mod\ p)$ with $r\in[0,p-1]$.
\end{lemma}
\begin{proof}
We construct the inverse isomorphism by defining an embedding $\psi:\bC[\tilde Q][s]\to\bC[\{(z_{i,j}^{(k)})^{\pm1}\}]$ such that for every $I^{(k)}\in\tilde\cI_{\pbw}$ with $I=(\alpha_1,\dots,\alpha_p)$ and $k=lp+r$ we have (recall Definitions~\ref{polytopedef} and~\ref{hibidef}):
\begin{equation}\label{pbwcov}
\psi:s\prod_{\substack{q\in(O\cap\varphi_{\pbw}(I^{(k)}))\cup\\(C\cap\max_{\prec}\varphi_{\pbw}(I^{(k)}))}} y_q\mapsto z_{1,\alpha_{p-r+1}}^{(l+1)}\dots z_{r,\alpha_p}^{(l+1)}z_{r+1,\alpha_1}^{(l)}\dots z_{p,\alpha_{p-r}}^{(l)}.
\end{equation}

We proceed to define $\psi$ by specifying its image on every $y_{i,j}$ and $s$. The first step will be to relabel  the variables $z_{\alpha,\beta}^{(k)}$, i.e.\ introduce symbols $\zeta_{i,j}$ each of which is identified with some $z_{\alpha,\beta}^{(k)}$. For readability the definition of $\psi$ is then given in terms of these $\zeta_{i,j}$. Example~\ref{infpbwashibiringex} below shows how the map $\psi$ works for a particular $I^{(k)}$.

Consider $q_{i,j}\in\tilde Q$ with $i=mp+t$ for $t\in[1,p]$ and set $\zeta_{i,j}=z_{t,j\ \mod\ n}^{(m)}$. Note that $q_{i,j}\mapsto\zeta_{i,j}$ is a bijection from $\tilde Q$ to the set of all $z_{i,j}^{(k)}$ except those with $k=0$ and $j\in[i,p]$. It can be helpful to note that if one orders the $\zeta_{i,j}$ first by $i$ and then by $j$ one obtains precisely the ordering of the variables $z_{i,j}^{(k)}$ considered in Definition~\ref{infpbwmonord} with the aforementioned ${p+1}\choose 2$ variables removed. We use the shorthand $y_{q_{i,j}}=y_{i,j}$ and define $\psi$ by 
\begin{itemize}
\item $\psi(y_{i,j})=\zeta_{i,j}\zeta_{i,i}^{-1}$ when $j>i$,
\item $\psi(y_{i,i})=\zeta_{i,i}\zeta_{i-p,i-p}^{-1}$ and
\item $\psi(s)=z_{1,1}^{(0)}\dots z_{p,p}^{(0)}$.
\end{itemize}

Now, for an $I^{(k)}$ as above let us describe the set of factors in the product on the left. We again consider \[I'=(\beta_1,\dots,\beta_p)=(\alpha_1-k\;\;\;\mod\ n,\dots,\alpha_p-k\;\;\;\mod\ n)\in\cI_{\pbw}.\] Let $L^{(k)}\subset Q^{(k)}$ be the antichain composed of elements $q_{i+k,\beta_i+k}$ with $\beta_i>p$. In the proof of Proposition~\ref{infpbwbijection} we have seen that $C\cap\max_{\prec}\varphi_{\pbw}(I^{(k)})$ is precisely $L^{(k)}$ while $O\cap\varphi_{\pbw}(I^{(k)})$ consists of all $q_{m,m}$ with $m\le k+p$ (including all $q_{i+k,\beta_i+k}$ for which $\beta_i=i$).

However, for $i\in[1,p-r]$ we have $\zeta_{i+k,\beta_i+k}=z_{r+i,\alpha_i}^{(l)}$ and for $i\in[p-r+1,p]$ we have $\zeta_{i+k,\beta_i+k}=z_{i+r-p,\alpha_i}^{(l+1)}$. It is easily seen from here that when $\psi$ is applied to the left-hand side of~\eqref{pbwcov} one indeed obtains the right-hand side and all the other factors cancel out.

To see that $\psi$ is injective one verifies that it is distinct on distinct monomials or, equivalently, that $\psi$ cannot map the quotient of two distinct monomials to 1. We show that $\psi(M)\neq 1$ for any Laurent monomial $M=s^d\prod y_{i,j}^{d_{i,j}}\neq 1$ with $d,d_{i,j}\in\bZ$. Indeed, if all $d_{i,j}=0$ the statement is trivial, consider that $d_{i,j}\neq0$ for which $q_{i,j}$ is $\prec$-maximal. Then $\psi(M)$ contains $\zeta_{i,j}$ in degree $d_{i,j}\neq 0$, since all other appearing $\zeta_{i',j'}$ satisfy $q_{i',j'}\prec q_{i,j}$.
\end{proof}

\begin{example}\label{infpbwashibiringex}
As in Example~\ref{pbw_ideal} consider $I^{(k)}=(1,4,6)^{(2)}$. The left-hand side of \eqref{pbwcov} is then equal to $sy_{4,4}y_{5,5}y_{3,8}y_{5,6}$. We have
\begin{align*}
\psi(y_{4,4})&=\zeta_{4,4}/\zeta_{1,1},\\
\psi(y_{5,5})&=\zeta_{5,5}/\zeta_{2,2},\\
\psi(y_{3,8})&=\zeta_{3,8}/\zeta_{3,3},\\
\psi(y_{5,6})&=\zeta_{5,6}/\zeta_{5,5},\\
\psi(s)&=z_{1,1}^{(0)}z_{2,2}^{(0)}z_{3,3}^{(0)}.
\end{align*}
We deduce that 
\begin{equation}\label{psiimage}
\psi(sy_{4,4}y_{5,5}y_{3,8}y_{5,6})=\frac{\zeta_{4,4}\zeta_{3,8}\zeta_{5,6}z_{1,1}^{(0)}z_{2,2}^{(0)}z_{3,3}^{(0)}}{\zeta_{1,1}\zeta_{2,2}\zeta_{3,3}}.
\end{equation}
However, by the definition of $\zeta_{i,j}$ we have $\zeta_{4,4}=z_{1,4}^{(1)}$, $\zeta_{3,8}=z_{3,1}^{(0)}$, $\zeta_{5,6}=z_{2,6}^{(1)}$, $\zeta_{1,1}=z_{1,1}^{(0)}$, $\zeta_{2,2}=z_{2,2}^{(0)}$ and $\zeta_{3,3}=z_{3,3}^{(0)}$. Hence, the right-hand side of~\eqref{psiimage} is equal to $z_{1,4}^{(1)}z_{3,1}^{(0)}z_{2,6}^{(1)}$ which is seen to coincide with the right-hand side of~\eqref{pbwcov}.
\end{example}

\begin{proof}[Proof of Theorem~\ref{main1}]
Consider an $\bN_0^2$-grading $\grad$ on $\bC[\{z_{i,j}^{(k)}\}]$ given by $\grad z_{i,j}^{(k)}=(1,k)$. Evidently, the subalgebras $\tilde P$, $\initial_{<_{\pbw}}\tilde P$, $\mathcal A$ and $\initial_{<_{\ss}}\tilde P$ are $\grad$-homogeneous with finite-dimensional components. We also write $\grad$ to denote an $\bN_0^2$-grading on $\bC[\tilde Q][s]$ given by $\grad s=(p,0)$, $\grad y_{i,i}=(0,1)$ and all other $\grad y_{i,j}=(0,0)$. The subalgebras $A_{O,C}(\tilde Q,\prec)$ and $A_{\tilde Q,\varnothing}(\tilde Q,\prec)$ are $\grad$-homogeneous and have the same Hilbert series with respect to $\grad$ in view of Corollary~\ref{hibibasis}.

Now, it is evident that the isomorphisms defined in Lemmas~\ref{infpbwashibiring} and~\ref{infssashibiring} respect the grading $\grad$. We deduce that $\mathcal A$ and $\initial_{<_{\ss}}\tilde P$ have the same Hilbert series. By Corollary~\ref{samesize}, however, $\initial_{<_{\pbw}}\tilde P$ (which contains $\mathcal A$) must have the same Hilbert series as $\initial_{<_{\ss}}\tilde P$ and we obtain $\mathcal A=\initial_{<_{\pbw}}\tilde P$.
\end{proof}

\begin{corollary}\label{newdegen}
The toric variety of the interpolating polytope $\Pi_{O,C}(\tilde Q,\prec)$ is a flat degeneration of the semi-infinite Grassmannian $\Proj\tilde\cP$.
\end{corollary}

\begin{remark}
Motivated by the finite case, the authors initially hoped to realize the toric variety of the chain polytope $\Pi_{\varnothing,\tilde Q}(\tilde Q,\prec)$ as a flat degeneration of $\Proj\tilde\cP$. Such attempts were, however, unsuccessful. This situation is somewhat reminiscent of the Feigin--Fourier--Littelmann degeneration of the complete flag variety whose toric ideal is not given by the chain polytope but by a certain different interpolating polytope of the respective poset (cf.\ \cite[Remark 6.10]{Makhlin}).
\end{remark}

We proceed to describe the corresponding basis in $\tilde\cP$.
\begin{definition}\label{precpbwdef}
For $I^{(k)},I'^{(k')}\in\tilde\cI_{\pbw}$ with $I=(\alpha_1,\dots,\alpha_p)$ and $I'=(\alpha'_1,\dots,\alpha'_p)$ set $I^{(k)}\preceq_{\pbw}I'^{(k')}$ if $k\le k'$ and for any $i\in[k'-k+1,p]$ there exists an $i'\in[i-(k'-k),p]$ such that \[k+(\alpha_i-k\;\;\;\mod\ n)\le k'+(\alpha'_{i'}-k'\;\;\;\mod\ n).\] Let $\prec_{\pbw}$ be the corresponding strict relation.
\end{definition}

\begin{example}\label{columns_comparison} One can easily verify that for elements from Example \ref{pbw_columns} we have $(1,2,6,5)^{(0)} \prec (3,8,7,6)^{(2)}$ and $(3,8,7,6)^{(2)} \prec (5,1,7,9)^{(4)}$. On the other hand, elements $(1,2,6,5)^{(0)}$ and $(2,3,4,5)^{(1)}$ are incomparable.
\end{example}

\begin{proposition}\label{infpbwposetisom}
$\varphi_{\pbw}$ is an isomorphism between the posets $(\tilde\cI_{\pbw},\prec_{\pbw})$ and $(\cJ(\tilde Q),\subset)$.
\end{proposition}
\begin{proof}
For $I^{(k)},I'^{(k')}$ as in Definition~\ref{precpbwdef} set $\beta_i=\alpha_i-k\ \ \mod\ n$ and $\beta'_i=\alpha'_i-k'\ \ \mod\ n$. According to Definition~\ref{infpbwbijection} we are to check that $I^{(k)}\preceq_{\pbw}I'^{(k')}$ if and only if for every $i\in[1,p]$ there exists an $i'\in[1,p]$ such that $q_{k+i,k+\beta_i}\preceq q_{k'+i',k'+\beta'_{i'}}$. If $i\le k'-k$ one may always choose $i'=p$, since then $k'+i'-(k+i)\ge p$. 

Suppose $i\ge k'-k+1$. According to Definition~\ref{precpbwdef} we check that there exists $i'\in[1,p]$ with $q_{k+i,k+\beta_i}\preceq q_{k'+i',k'+\beta'_{i'}}$ if and only if there exists  $i''\in[i-(k'-k),p]$ with $k+\beta_i\le k'+\beta'_{i''}$. If $\beta_i\le p+(k'-k)$, then for $i'=p$ we have $k+\beta_i\le k'+\beta'_{i'}$ and $k+i\le k'+i'$ which implies $q_{k+i,k+\beta_i}\preceq q_{k'+i',k'+\beta'_{i'}}$. If $\beta_i>p+(k'-k)$, then $q_{k+i,k+\beta_i}\preceq q_{k'+i',k'+\beta'_{i'}}$ if and only if 
\[
i'\in[i-(k'-k),p]\quad \text{ and }\quad k+\beta_i\le k'+\beta'_{i'}
\]
because $k'+i'-k-i\ge p$ and $\beta_{i'}+k'-\beta_i-k\ge n-p$ are impossible.
\end{proof}

In particular, we see that $(\tilde\cI_{\pbw},\prec_{\pbw})$ is isomorphic to $(\tilde\cI_{\ss},\prec_{\ss})$ with an isomorphism given by $\varphi_{\ss}^{-1}\varphi_{\pbw}$.

Next, recall Definition~\ref{tableaudef} and consider a semi-infinite tableau $T$ of shift $(k_1,\dots,k_m)$ with the content of box $(i,j)$ equal to $T_{i,j}$.
\begin{definition}
We say that $T$ is \textit{PBW-semistandard} if
\begin{itemize}
\item $(T_{i,k_i+1}-k_i,\dots,T_{i,k_i+p}-k_i)$ is a PBW tuple for every $i$,
\item for every box $(i,j)$ in $T$ such that $(i+1,j)$ also lies in $T$ there exists a box $(i+1,j')$ in $T$ with $j'\ge j$ such that $T_{i+1,j'}\ge T_{i,j}$.
\end{itemize}
\end{definition}
\begin{example}
The following is an example of a PBW-semistandard semi-infinite tableau of shift $(0,2,4)$. As in Example 4.14, the yellow boxes depict the height at which its columns are placed. Note that the columns of this tableau appeared in Example \ref{pbw_columns} and Example \ref{columns_comparison}.
{
\begin{center}
    \begin{ytableau}
    \none & \none & 9\\
    \none & \none & 7\\
       \none & 6 & 1 \\
       \none & 7 & 5 \\
       5 & 8 & *(yellow)\\
       6 & 3 & *(yellow)\\
       2 & *(yellow) & *(yellow)\\
       1 & *(yellow) & *(yellow)
\end{ytableau}
\end{center}
}
\end{example}

Let us denote \[I_i(T)=(T_{i,k_i+1}\ \mod\ n,\dots,T_{i,k_i+p}\ \mod\ n).\] It is easily seen that  $T$ is PBW-semistandard if and only if every $T_{i,j}\in[k_i+1,k_i+p]$, every $I_i(T)^{(k_i)}\in\tilde\cI_{\pbw}$ and \[I_1(T)^{(k_1)}\preceq_{\pbw}\dots\preceq_{\pbw}I_m(T)^{(k_m)}.\] It is also evident that for any $I_1^{(k_1)},\dots,I_m^{(k_m)}\in\tilde\cI_{\pbw}$ such that \[I_1^{(k_1)}\preceq_{\pbw}\dots\preceq_{\pbw}I_m^{(k_m)}\] there exists a unique semi-infinite tableau $T'$ of shift $(k_1,\dots,k_m)$ satisfying $I_i(T')=I_i$ for all $i$. Indeed, if $I_i=(\alpha^i_1,\dots \alpha^i_p)$, the content of $(i,k_i+j)$ in $T'$ is equal to $k_i+(\alpha^i_j-k_i\ \ \mod\ n)$. Similarly to the previous sections we obtain the following.
\begin{theorem}\label{main2}
The set of all products $D_{I_1(T)}^{(k_1)}\dots D_{I_m(T)}^{(k_m)}$ with $k_1\le\dots\le k_m$ and $T$ a PBW-semistandard semi-infinite tableau of shift $(k_1,\dots,k_m)$ is a basis in the semi-infinite Pl\"ucker algebra $\tilde\cP$.
\end{theorem}

\begin{remark}
For each of the four degeneration constructions discussed in this paper one could use the language of initial ideals and Gr\"obner degenerations instead of the language of initial subalgebras and sagbi degenerations. Let us outline how this would work for the construction in this section. Let $\mathcal S$ be the polynomial ring in variables $X_I^{(k)}$ with $I\in\tilde\cI_{\pbw}$ and $k\in\bN_0$. Consider a map $\theta:\mathcal S\to\bC[\{z_{i,j}^{(k)}\}]$ given by $\theta(X_I^{(k)})=D_I^{(k)}$. The image of $\theta$ is $\tilde\cP$, the kernel of $\theta$ is the \textit{semi-infinite Pl\"ucker ideal} of relations in $\tilde\cP$. Now let us also consider the map $\eta:\mathcal S\to\bC[\{z_{i,j}^{(k)}\}]$ given by $\eta(X_I^{(k)})=\initial_{<_{\pbw}}D_I^{(k)}$. Due to general Gr\"obner theoretic considerations, Theorem~\ref{main1} is equivalent to the statement that $\ker\eta$ is an initial ideal of $\ker\theta$ (with respect to a weight vector or to a partial monomial order). Furthermore, in these terms, the products $X_{I_1}^{(k_1)}\dots X_{I_m}^{(k_m)}$ for which $I_1^{(k_1)},\dots,I_m^{(k_m)}$ are \textbf{not} pairwise comparable with respect to $\prec_{\pbw}$ (so not forming a weakly-increasing tuple) span a monomial ideal in $\mathcal S$. Theorem~\ref{main2} can then be proved by showing that this monomial ideal is an initial ideal of $\ker\eta$ (and thus $\ker\theta$), which is a general property of interpolating polytopes (see~\cite[Proposition 5.3]{Makhlin}).
\end{remark}

\begin{remark}
As noted in the introduction, semi-infinite Grassmannians have also been studied due to their connection with the representation theory of current algebras $\mathfrak{sl}_n[t]$. Namely, there is an important class of cyclic representations of $\mathfrak{sl}_n[t]$ called global Weyl modules $W(\lambda)$ which are parametrized by dominant weights $\lambda$ of $\mathfrak{sl}_n$. It is proven in \cite{FM2} that the $k$-th homogeneous component of the algebra $\tilde\cP$ is the restricted dual of the module $W(k \omega_p)$. In this context Theorem \ref{main2} provides a new combinatorial formula for the characters of global Weyl modules analogous to the expression of Schur polynomials as generating functions of semistandard tableaux.
\end{remark}

\end{document}